\documentclass{amsart}
\usepackage{amssymb}
\usepackage{amsfonts}

\newtheorem{theorem}{Theorem}
\theoremstyle{plain}

\newtheorem{corollary}{Corollary}

\newtheorem{definition}{Definition}
\newtheorem{example}{Example}

\newtheorem{proposition}{Proposition}
\newtheorem{remark}{Remark}

\numberwithin{equation}{section}
\input{tcilatex}

\begin{document}
\title[Some Inequalities for $(h-s)_{1,2}-$convex Functions]{Some New
Inequalities for $(h-s)_{1,2}-$convex Functions via Further Properties}
\author{$^{\blacksquare }$M. Emin \"{O}zdemir}
\address{$^{\blacksquare }$Ataturk University, K.K. Education Faculty,
Department of Mathematics, 25240, Erzurum, Turkey}
\email{emos@atauni.edu.tr}
\author{$^{\bigstar }$Ahmet Ocak Akdemir}
\address{$^{\bigstar }$A\u{g}r\i\ \.{I}brahim \c{C}e\c{c}en University,
Faculty of Science and Letters, Department of Mathematics, 04100, A\u{g}r\i
, Turkey}
\email{ahmetakdemir@agri.edu.tr}
\author{$^{\triangle }$Mevl\"{u}t TUN\c{C}}
\address{$^{\triangle }$Kilis 7 Aral\i k University, Faculty of Science and
Arts, Department of Mathematics, 79000, Kilis, Turkey}
\email{mevluttunc@kilis.edu.tr} 

\begin{abstract}
In this paper, we establish some new inequalities of the
Hermite-Hadamard like for class of $(h-s)_{1,2}-$convex functions
which are ordinary, super-multiplicative or similarly ordered and
nonnegative.
\end{abstract}

\maketitle

\section{INTRODUCTION}

Let $f:I\subseteq
\mathbb{R}
\rightarrow
\mathbb{R}
$ be a convex mapping and $a,b\in I$ with $a<b$. The following double
inequality:

\begin{equation}
f\left( \frac{a+b}{2}\right) \leq \frac{1}{b-a}\int_{a}^{b}f\left( x\right)
dx\leq \frac{f\left( a\right) +f\left( b\right) }{2}  \label{1.1}
\end{equation}%
is well-known in the literature as Hadamard's inequality for convex mapping.
Note that some of the classical inequalities for means can be derived from (%
\ref{1.1}) for appropriate particular selections of the mapping $f$. Both
inequalities hold in the reversed direction if $f$ is concave.

\begin{definition}
\lbrack See \cite{GO}] We say that $f:I\rightarrow
\mathbb{R}
$\ is Godunova-Levin function or that $f$\ belongs to the class $Q\left(
I\right) $ if $f$\ is non-negative and for all $x,y\in I$\ and $t\in \left(
0,1\right) $ we have \ \ \ \ \ \ \ \ \ \ \ \ \
\begin{equation}
f\left( tx+\left( 1-t\right) y\right) \leq \frac{f\left( x\right) }{t}+\frac{%
f\left( y\right) }{1-t}.  \label{1.2}
\end{equation}
\end{definition}

\begin{definition}
\lbrack See \cite{HU}] Let $s\in \left( 0,1\right] .$\ A function $f:\left(
0,\infty \right] \rightarrow \left[ 0,\infty \right] $ is said to be $s-$%
convex in the second sense if \ \ \ \ \ \ \ \ \ \ \ \
\begin{equation}
f\left( tx+\left( 1-t\right) y\right) \leq t^{s}f\left( x\right) +\left(
1-t\right) ^{s}f\left( y\right)  \label{1.3}
\end{equation}%
for all $x,y\in \left( 0,b\right) $ and $t\in \left[ 0,1\right] .$
\end{definition}

In 1978, Breckner introduced $s-$convex functions as a generalization of
convex functions in \cite{WW}. Also, in that one work Breckner proved the
important fact that the set valued map is $s-$convex only if the associated
support function is $s-$convex function\ in \cite{WW2}. A number of
properties and connections with $s-$convex in the first sense are discussed
in paper \cite{HU}. Of course, $s-$convexity means just convexity when $s=1$.

\begin{definition}
\lbrack See \cite{SS1}] We say that $f:I\subseteq
\mathbb{R}
\rightarrow
\mathbb{R}
$ is a $P-$ function or that $f$ belongs to the class $P\left( I\right) $ if
$f$ is nonnegative and for all $x,y\in I$ and $t\in \left[ 0,1\right] ,$\ we
have%
\begin{equation}
f\left( tx+\left( 1-t\right) y\right) \leq f\left( x\right) +f\left(
y\right) .  \label{1.4}
\end{equation}
\end{definition}

\begin{definition}
\lbrack See \cite{SA}] Let $h:J\subseteq
\mathbb{R}
\rightarrow
\mathbb{R}
$\ be a non-negative function. We say that $f:I\subseteq
\mathbb{R}
\rightarrow
\mathbb{R}
$ is an $h-$convex function or that $f$ belongs to the class $SX\left(
h,I\right) $, if $f$ is nonnegative and for all $x,y\in I$\ and $\alpha \in %
\left[ 0,1\right] $\ we have \ \ \ \ \ \ \ \ \ \ \ \
\begin{equation}
f\left( \alpha x+\left( 1-\alpha \right) y\right) \leq h\left( \alpha
\right) f\left( x\right) +h\left( 1-\alpha \right) f\left( y\right) .
\label{1.5}
\end{equation}
\end{definition}

If inequality (\ref{1.5}) is reversed, then $f$ is said to be $h-$concave,
i.e., $f\in SV\left( h,I\right) $. Obviously, if $h\left( \alpha \right)
=\alpha $, then all nonnegative convex functions belong to $SX\left(
h,I\right) $\ and all nonnegative concave functions belong to $SV(h,I);$ if $%
h(\alpha )=\frac{1}{\alpha },$ then $SX(h,I)=Q(I);$ if $h(\alpha )=1,$ $%
SX(h,I)\supseteq P(I);$ and if $h(\alpha )=\alpha ^{s},$ where $s\in \left(
0,1\right) ,$ then $SX(h,I)\supseteq K_{s}^{2}.$

Furthermore, in \cite{SA2} Bombardelli and Varo\v{s}anec wrote some
generalizations of the Hermite-Hadamard inequalities and some properties of
functions $H$ and $F.$

More about those inequalities can be found in a number of papers (for
example: see \cite{SA, OZ, SAR, BU, AO, METU}).

\begin{definition}
\lbrack See \cite{MMA}] Let $h:J\subset
\mathbb{R}
\rightarrow
\mathbb{R}
$ \ be a non-negative function, $h\neq 0.$We say that $f:%
\mathbb{R}
^{+}\cup \left\{ 0\right\} \rightarrow
\mathbb{R}
$ is an $\left( h-s\right) _{1}-$convex function in the first sense, or that
$f$ belong to the class $SX(\left( h-s\right) _{1},I),$ if $f$ is
non-negative and for all $x,y\in \left[ 0,\infty \right) =I,$ $s\in \left(
0,1\right] ,$ $t\in \left[ 0,1\right] $ we have%
\begin{equation}
f(tx+(1-t)y)\leq h^{s}(t)f(x)+(1-h^{s}(t))f(y).  \label{21}
\end{equation}%
\bigskip If inequality (\ref{21}) is reversed, then $f$ is said to be $%
\left( h-s\right) _{1}-$concave function in the first sense, i.e., $f\in
SV(\left( h-s\right) _{1},I).$\bigskip
\end{definition}

\begin{definition}
\lbrack See \cite{MMA}] Let $\ \ \ \ h:J\subset
\mathbb{R}
\rightarrow
\mathbb{R}
$ be a non-negative function,\ $h\neq 0.$ We say that $f:%
\mathbb{R}
^{+}\cup \left\{ 0\right\} \rightarrow
\mathbb{R}
$ \ \ is an \ $(h-s)_{2}-$convex function in the second sense, or that $f$
belong to the class $SX(\left( h-s\right) _{2},I)$ , if $f$ is non-negative
and for all $u,v\in \left[ 0,\infty \right) =I,$ $s\in \left( 0,1\right] ,$ $%
t\in \left[ 0,1\right] $ we have%
\begin{equation}
f(tu+(1-t)v)\leq h^{s}(t)f(u)+h^{s}(1-t)f(v).  \label{22}
\end{equation}%
If inequality (\ref{22}) is reversed, then $f$ is said to be $(h-s)_{2}-$%
concave function in the second sense, i.e., $f\in SV(\left( h-s\right)
_{2},I).$\bigskip
\end{definition}

Obviously, in (\ref{22}), if $h(t)=t$, then all $s-$convex functions in the
second sense belongs to $SX(\left( h-s\right) _{2},I)$ and all $s-$concave
functions in the second sense belongs to $SV(\left( h-s\right) _{2},I)$, and
it can be easily seen that for $h(t)=t,$ $s=1,$ $(h-s)_{2}-$convexity
reduces to ordinary convexity defined on $\left[ 0,\infty \right) .$
Similarly, in (\ref{21}), if $h(t)=t$, then all $s-$convex functions in the
first sense belongs to $SX(\left( h-s\right) _{1},I)$ and all $s-$concave
functions in the first sense belongs to $SV(\left( h-s\right) _{1},I)$, and
it can be easily seen that for $h(t)=t$ $,s=1,$ $(h-s)_{1}-$convexity
reduces to ordinary convexity defined on $\left[ 0,\infty \right) .$

\begin{example}
\lbrack See \cite{MMA}] Let $h(t)=t$ be a function and let the function $f$
be defined as following;%
\begin{equation*}
f:[2,4]\rightarrow
\mathbb{R}
^{+},\text{ \ \ \ \ }f(x)=\ln x.
\end{equation*}%
Then $f$ is non-convex and non-$h-$convex function, but it is $(h-s)_{2}-$%
convex function.
\end{example}

\begin{definition}
\lbrack See \cite{SA}] A function $h:J\rightarrow
\mathbb{R}
$ is said to be a super-multiplicative function if%
\begin{equation}
h(xy)\geq h(x)h(y)  \label{1.6}
\end{equation}%
for all $x,y\in J$.

If inequality (\ref{1.6}) is reversed, then $h$ is said to be a
sub-multiplicative function. If the equality holds in (\ref{1.6}),
then $h$ is said to be a multiplicative function.
\end{definition}

\begin{definition}
\lbrack See \cite{HA}] A function $h:J\rightarrow
\mathbb{R}
$ is said to be a super-additive function if%
\begin{equation}
h(x+y)\geq h(x)+h(y)  \label{1.7}
\end{equation}%
for all $x,y\in J$.
\end{definition}

\begin{definition}
\lbrack See \cite{A}] Two functions $f:X\rightarrow
\mathbb{R}
$ and $g:X\rightarrow
\mathbb{R}
$ are said to be similarly ordered, shortly $f$ s.o. $g$, if%
\begin{equation*}
(f(x)-f(y))(g(x)-g(y))\geq 0
\end{equation*}%
for every $x,y\in X.$
\end{definition}

\begin{theorem}
\lbrack See \cite{SAR}] Let $f\in SX(h_{1},I),$ $g\in SX(h_{2},I),$ $a,b\in
I,$ $a<b,$ be functions such that $fg\in L_{1}\left( \left[ a,b\right]
\right) ,$ and $h_{1}h_{2}\in L_{1}\left( \left[ 0,1\right] \right) ,$ then
the following inequality holds%
\begin{equation*}
\frac{1}{b-a}\dint\limits_{a}^{b}f(x)g(x)\leq
M(a,b)\dint\limits_{0}^{1}h_{1}(t)h_{2}(t)dt+N(a,b)\dint%
\limits_{0}^{1}h_{1}(t)h_{2}(1-t)dt
\end{equation*}%
where%
\begin{eqnarray*}
M(a,b) &=&f(a)g(a)+f(b)g(b) \\
N(a,b) &=&f(a)g(b)+f(b)g(a).
\end{eqnarray*}
\end{theorem}

Motivated by the information given above, main purpose of this paper is to
give some inequalities under the special assumptions of $h-$convex functions
by using fairly elementary analysis. We also give some applications to
special means. Throughout the paper we will imply $M(a,b)=f(a)g(a)+f(b)g(b)$
and $N(a,b)=f(a)g(b)+f(b)g(a).$

\section{MAIN RESULTS}

We will start with the following result for $\left( h-s\right) _{1}-$convex
functions.

\begin{theorem}
\label{t1}Let $f,g\in SX(\left( h-s\right) _{1},I)$ and $h$ is a positive
function. If $fg\in L_{1}\left[ a,b\right] ,$ then we have the following
inequality;%
\begin{eqnarray}
&&\frac{1}{b-a}\dint\limits_{a}^{b}f^{\frac{x-a}{b-a}}(x)g^{\frac{b-z}{b-a}%
}(x)dx  \label{2.1} \\
&\leq &\frac{f(b)}{\left( b-a\right) ^{2}}\dint\limits_{a}^{b}\left(
x-a\right) h^{s}\left( \frac{x-a}{b-a}\right) dx+\frac{f(a)}{\left(
b-a\right) ^{2}}\dint\limits_{a}^{b}\left( x-a\right) \left( 1-h^{s}\left(
\frac{x-a}{b-a}\right) \right) dx  \notag \\
&&+\frac{g(b)}{\left( b-a\right) ^{2}}\dint\limits_{a}^{b}\left( b-x\right)
h^{s}\left( \frac{x-a}{b-a}\right) dx+\frac{g(a)}{\left( b-a\right) ^{2}}%
\dint\limits_{a}^{b}\left( b-x\right) \left( 1-h^{s}\left( \frac{x-a}{b-a}%
\right) \right) dx  \notag
\end{eqnarray}%
for all $s\in \left( 0,1\right] .$
\end{theorem}

\begin{proof}
Since $f,g$ are $\left( h-s\right) _{1}-$convex functions on $I$, we have%
\begin{eqnarray*}
f\left( tb+(1-t)a\right) &\leq &h^{s}(t)f\left( b\right) +\left(
1-h^{s}(t)\right) f\left( a\right) \\
g\left( tb+(1-t)a\right) &\leq &h^{s}(t)g\left( b\right) +\left(
1-h^{s}(t)\right) g\left( a\right) .
\end{eqnarray*}%
For $t\in \left[ 0,1\right] ,$ we can write%
\begin{eqnarray*}
f^{t}\left( tb+(1-t)a\right) &\leq &\left[ h^{s}(t)f\left( b\right) +\left(
1-h^{s}(t)\right) f\left( a\right) \right] ^{t} \\
g^{1-t}\left( tb+(1-t)a\right) &\leq &\left[ h^{s}(t)g\left( b\right)
+\left( 1-h^{s}(t)\right) g\left( a\right) \right] ^{1-t}.
\end{eqnarray*}%
By multiplying the above inequalities, we get
\begin{eqnarray}
&&f^{t}\left( tb+(1-t)a\right) g^{1-t}\left( tb+(1-t)a\right)  \label{1} \\
&\leq &\left[ h^{s}(t)f\left( b\right) +\left( 1-h^{s}(t)\right) f\left(
a\right) \right] ^{t}\left[ h^{s}(t)g\left( b\right) +\left(
1-h^{s}(t)\right) g\left( a\right) \right] ^{1-t}.  \notag
\end{eqnarray}%
Recall the General Cauchy Inequality (see \cite{T}, Theorem 3.1), let $%
\alpha $ and $\beta $ be positive real numbers satisfying $\alpha +\beta =1$%
. Then for every positive real numbers $x$ and $y$, we always have%
\begin{equation*}
\alpha x+\beta y\geq x^{\alpha }y^{\beta }.
\end{equation*}%
By applying General Cauchy Inequality to the right hand side of the
inequality (\ref{1}), we have%
\begin{eqnarray*}
&&f^{t}\left( tb+(1-t)a\right) g^{1-t}\left( tb+(1-t)a\right) \\
&\leq &t\left[ h^{s}(t)f\left( b\right) +\left( 1-h^{s}(t)\right) f\left(
a\right) \right] +(1-t)\left[ h^{s}(t)g\left( b\right) +\left(
1-h^{s}(t)\right) g\left( a\right) \right] .
\end{eqnarray*}%
By integrating the resulting inequality with respect to $t,$ over $\left[ 0,1%
\right] ,$ we obtain%
\begin{eqnarray*}
&&\int\limits_{0}^{1}f^{t}\left( tb+(1-t)a\right) g^{1-t}\left(
tb+(1-t)a\right) dt \\
&\leq &\int\limits_{0}^{1}t\left[ h^{s}(t)f\left( b\right) +\left(
1-h^{s}(t)\right) f\left( a\right) \right] dt+\int\limits_{0}^{1}(1-t)\left[
h^{s}(t)g\left( b\right) +\left( 1-h^{s}(t)\right) g\left( a\right) \right]
dt.
\end{eqnarray*}%
Computing the above integrals and by changing of the variable
$tb+(1-t)a=x,$ $(b-a)dt=dx$, we get the desired result.
\end{proof}

\begin{corollary}
\bigskip (1) In Theorem \ref{t1}, if we choose $x=a$, we get%
\begin{equation*}
\frac{1}{b-a}\dint\limits_{a}^{b}g\left( x\right) dx\leq g\left( b\right)
h^{s}\left( 0\right) +g(a)\left( 1-h^{s}\left( 0\right) \right)
\end{equation*}%
(2) In Theorem \ref{t1}, if we choose $x=b$, we get%
\begin{equation*}
\frac{1}{b-a}\dint\limits_{a}^{b}f\left( x\right) dx\leq f\left( b\right)
h^{s}\left( 1\right) +f\left( a\right) \left( 1-h^{s}\left( 1\right) \right)
\end{equation*}%
(3) In Theorem \ref{t1}, if we choose $x=\frac{a+b}{2}$, we get%
\begin{equation*}
\frac{1}{b-a}\dint\limits_{a}^{b}\sqrt{f\left( x\right) g\left( x\right) }%
dx\leq \frac{f\left( b\right) +g\left( b\right) }{2}h^{s}\left( \frac{1}{2}%
\right) +\frac{f\left( a\right) +g\left( a\right) }{2}\left( 1-h^{s}\left(
\frac{1}{2}\right) \right) .
\end{equation*}
\end{corollary}

A similar result will obtain in the following Theorem for $\left( h-s\right)
_{2}-$convex functions.

\begin{theorem}
\label{t2}Let $f,g\in SX(\left( h-s\right) _{2},I)$ and $h$ is a positive
function. If $fg\in L_{1}\left[ a,b\right] ,$ then we have the following
inequality;%
\begin{eqnarray}
&&\frac{1}{b-a}\dint\limits_{a}^{b}f^{\frac{x-a}{b-a}}(x)g^{\frac{b-z}{b-a}%
}(x)dx  \label{2.2} \\
&\leq &\frac{f(b)}{\left( b-a\right) ^{2}}\dint\limits_{a}^{b}\left(
x-a\right) h^{s}\left( \frac{x-a}{b-a}\right) dx+\frac{f(a)}{\left(
b-a\right) ^{2}}\dint\limits_{a}^{b}\left( x-a\right) h^{s}\left( \frac{b-x}{%
b-a}\right) dx  \notag \\
&&+\frac{g(b)}{\left( b-a\right) ^{2}}\dint\limits_{a}^{b}\left( b-x\right)
h^{s}\left( \frac{x-a}{b-a}\right) dx+\frac{g(a)}{\left( b-a\right) ^{2}}%
\dint\limits_{a}^{b}\left( b-x\right) h^{s}\left( \frac{b-x}{b-a}\right) dx
\notag
\end{eqnarray}%
for all $s\in \left( 0,1\right] .$
\end{theorem}

\begin{proof}
The proof is immediately follows from the proof of the above Theorem, but
now $f,g$ are $\left( h-s\right) _{2}-$convex functions on $I$.
\end{proof}

\bigskip

\begin{corollary}
(1)In Theorem \ref{t2}, if we choose $x=a$, we get%
\begin{equation*}
\frac{1}{b-a}\dint\limits_{a}^{b}g\left( x\right) dx\leq g\left( b\right)
h^{s}\left( 0\right) +g\left( a\right) h^{s}\left( 1\right)
\end{equation*}%
(2) In Theorem \ref{t2}, if we choose $x=b$, we get%
\begin{equation*}
\frac{1}{b-a}\dint\limits_{a}^{b}f\left( x\right) dx\leq f\left( b\right)
h^{s}\left( 1\right) +f\left( a\right) h^{s}\left( 0\right)
\end{equation*}%
(3) In Theorem \ref{t2}, if we choose $x=\frac{a+b}{2}$, we get%
\begin{equation*}
\frac{1}{b-a}\dint\limits_{a}^{b}\sqrt{f\left( x\right) g\left( x\right) }%
dx\leq \frac{1}{2}h^{s}\left( \frac{1}{2}\right) \left( f\left( a\right)
+f\left( b\right) +g\left( a\right) +g\left( b\right) \right) .
\end{equation*}
\end{corollary}

\begin{theorem}
\label{t3}Let $f,g\in SX(\left( h-s\right) _{1},I)$, $h$ is
super-multiplicative on $I$. If $f,g,fg\in L_{1}\left[ a,b\right] $
and $h\in
L_{1}\left[ 0,1\right] ,$ then we have the following inequality;%
\begin{eqnarray}
&&  \label{2.3} \\
&&\frac{1}{\left( b-a\right) ^{2}}\dint\limits_{a}^{b}\dint\limits_{a}^{b}%
\dint\limits_{0}^{1}f\left( tx+\left( 1-t\right) y\right) g(tx+\left(
1-t\right) y)dtdxdy  \notag \\
&\leq &\left( \frac{1}{b-a}\dint\limits_{a}^{b}f(x)g(x)dx\right) \left(
\dint\limits_{0}^{1}h^{s}(t^{2})dt\right) +\left( \frac{1}{b-a}%
\dint\limits_{a}^{b}f(y)g(y)dy\right) \left( \dint\limits_{0}^{1}\left(
1-h^{s}(t)\right) ^{2}dt\right)  \notag \\
&&+\left( \frac{1}{b-a}\dint\limits_{a}^{b}\dint\limits_{a}^{b}\left[
f(y)g(x)+f(x)g(y)\right] dxdy\right) \left( \dint\limits_{0}^{1}\left(
h^{s}(t)-h^{s}(t^{2})\right) dt\right)  \notag
\end{eqnarray}%
for all $s\in \left( 0,1\right] ,$ $x,y\in I\subseteq
\mathbb{R}
.$
\end{theorem}

\begin{proof}
Since $f,g$ are $\left( h-s\right) _{1}-$convex functions on $I$, we have%
\begin{eqnarray*}
f\left( tx+\left( 1-t\right) y\right) &\leq &h^{s}(t)f(x)+\left(
1-h^{s}\left( t\right) \right) f(y) \\
g(tx+\left( 1-t\right) y) &\leq &h^{s}(t)g(x)+\left( 1-h^{s}\left( t\right)
\right) g(y).
\end{eqnarray*}%
By multiplying the above inequalities and since $h$ is
super-multiplicative
function, we get%
\begin{eqnarray*}
&&f\left( tx+\left( 1-t\right) y\right) g(tx+\left( 1-t\right) y) \\
&\leq &h^{s}(t^{2})f(x)g(x)+\left( 1-h^{s}(t)\right) ^{2}f(y)g(y) \\
&&+\left( h^{s}(t)-h^{s}(t^{2})\right) \left[ f(y)g(x)+f(x)g(y)\right] .
\end{eqnarray*}%
By integrating the resulting inequality with respect to $t$ over $\left[ 0,1%
\right] $ and with respect to $x,y$ over $\left[ a,b\right] \times \left[ a,b%
\right] ,$ we have%
\begin{eqnarray*}
&&\frac{1}{b-a}\dint\limits_{a}^{b}\dint\limits_{a}^{b}\dint\limits_{0}^{1}f%
\left( tx+\left( 1-t\right) y\right) g(tx+\left( 1-t\right) y)dtdxdy \\
&\leq &\left( \dint\limits_{a}^{b}f(x)g(x)dx\right) \left(
\dint\limits_{0}^{1}h^{s}(t^{2})dt\right) +\left(
\dint\limits_{a}^{b}f(y)g(y)dy\right) \left( \dint\limits_{0}^{1}\left(
1-h^{s}(t)\right) ^{2}dt\right) \\
&&+\dint\limits_{a}^{b}\dint\limits_{a}^{b}\left[ f(y)g(x)+f(x)g(y)\right]
dxdy\dint\limits_{0}^{1}\left( h^{s}(t)-h^{s}(t^{2})\right) dt
\end{eqnarray*}%
By dividing $\frac{1}{b-a},$ the proof is completed.
\end{proof}

\begin{remark}
\bigskip In Theorem \ref{t3}, if we take $h\left( t\right) =t$ and $s=1,$
then we get%
\begin{eqnarray*}
&&\frac{1}{\left( b-a\right) ^{2}}\dint\limits_{a}^{b}\dint\limits_{a}^{b}%
\dint\limits_{0}^{1}f\left( tx+\left( 1-t\right) y\right) g(tx+\left(
1-t\right) y)dtdxdy \\
&\leq &\frac{1}{3\left( b-a\right) }\dint\limits_{a}^{b}f(x)g(x)dx+\frac{1}{%
3\left( b-a\right) }\dint\limits_{a}^{b}f(y)g(y)dy+\frac{1}{6\left(
b-a\right) }\dint\limits_{a}^{b}\dint\limits_{a}^{b}N\left( x,y\right) dxdy.
\end{eqnarray*}
\end{remark}

\begin{theorem}
\label{t4}Let $f,g\in SX(\left( h-s\right) _{2},I)$, $h$ is
super-multiplicative and $f,g$ be similarly ordered functions on $I$. If $%
f,g,fg\in L_{1}\left[ a,b\right] $ and $h\in L_{1}\left[ 0,1\right] ,$ then
we have the following inequality;%
\begin{eqnarray*}
&&\frac{1}{\left( b-a\right) ^{2}}\dint\limits_{a}^{b}\dint\limits_{a}^{b}%
\dint\limits_{0}^{1}f\left( tx+\left( 1-t\right) y\right) g(tx+\left(
1-t\right) y)dtdxdy \\
&\leq &\left( \frac{1}{b-a}\dint\limits_{a}^{b}f(x)g(x)dx\right) \left(
\dint\limits_{0}^{1}\left( h^{s}(t^{2})+h^{s}(t-t^{2})\right) dt\right) \\
&&+\left( \frac{1}{b-a}\dint\limits_{a}^{b}f(y)g(y)dy\right) \left(
\dint\limits_{0}^{1}\left( h^{s}\left( \left( 1-t\right) ^{2}\right)
+h^{s}(t-t^{2})\right) dt\right) \\
&=&\left( \frac{1}{b-a}\dint\limits_{a}^{b}f(x)g(x)dx\right) \left(
\dint\limits_{0}^{1}\left( h^{s}(t)+h^{s}(1-t)\right) ^{2}dt\right)
\end{eqnarray*}%
for all $s\in \left( 0,1\right] ,$ $x,y\in I\subseteq
\mathbb{R}
.$
\end{theorem}

\begin{proof}
By a similar argument to the proof of the previous Theorem, since $f,g$ are $%
\left( h-s\right) _{2}-$convex functions on $I$, we have%
\begin{eqnarray*}
f\left( tx+\left( 1-t\right) y\right) &\leq &h^{s}(t)f(x)+h^{s}\left(
1-t\right) f(y) \\
g(tx+\left( 1-t\right) y) &\leq &h^{s}(t)g(x)+h^{s}\left( 1-t\right) g(y).
\end{eqnarray*}%
By multiplying the above inequalities and since $h$ is
super-multiplicative
function, we get%
\begin{eqnarray*}
&&f\left( tx+\left( 1-t\right) y\right) g(tx+\left( 1-t\right) y) \\
&\leq &h^{s}(t^{2})f(x)g(x)+h^{s}\left( \left( 1-t\right) ^{2}\right)
f(y)g(y) \\
&&+h^{s}(t-t^{2})\left[ f(y)g(x)+f(x)g(y)\right] .
\end{eqnarray*}%
Since $f$ and $g$ are similarly ordered functions, we get%
\begin{eqnarray*}
&&f\left( tx+\left( 1-t\right) y\right) g(tx+\left( 1-t\right) y) \\
&\leq &\left( h^{s}(t^{2})+h^{s}(t-t^{2})\right) f(x)g(x)+\left( h^{s}\left(
\left( 1-t\right) ^{2}\right) +h^{s}(t-t^{2})\right) f(y)g(y).
\end{eqnarray*}%
By integrating the resulting inequality with respect to $t$ over $\left[ 0,1%
\right] $ and with respect to $x,y$ over $\left[ a,b\right] \times \left[ a,b%
\right] ,$ we have%
\begin{eqnarray*}
&&\frac{1}{b-a}\dint\limits_{a}^{b}\dint\limits_{a}^{b}\dint\limits_{0}^{1}f%
\left( tx+\left( 1-t\right) y\right) g(tx+\left( 1-t\right) y)dtdxdy \\
&\leq &\left( \dint\limits_{a}^{b}f(x)g(x)dx\right) \left(
\dint\limits_{0}^{1}\left( h^{s}(t^{2})+h^{s}(t-t^{2})\right) dt\right) \\
&&+\left( \dint\limits_{a}^{b}f(y)g(y)dy\right) \left(
\dint\limits_{0}^{1}\left( h^{s}\left( \left( 1-t\right) ^{2}\right)
+h^{s}(t-t^{2})\right) dt\right) \\
&=&\left( \dint\limits_{a}^{b}f(x)g(x)dx\right) \left(
\dint\limits_{0}^{1}\left( h^{s}(t)+h^{s}(1-t)\right) ^{2}dt\right) .
\end{eqnarray*}%
By dividing $\frac{1}{b-a},$ the proof is completed.
\end{proof}

\begin{theorem}
\label{t5}Let $f,g\in SX(\left( h-s\right) _{2},I)$, $h$ is
super-multiplicative and $f,g$ be similarly ordered functions on $I$. If $%
f,g,fg\in L_{1}\left[ a,b\right] $ and $h\in L_{1}\left[ 0,1\right] ,$ then
we have the following inequality;%
\begin{eqnarray*}
&&\frac{1}{\left( b-a\right) ^{2}}\dint\limits_{a}^{b}\dint\limits_{a}^{b}%
\dint\limits_{0}^{1}f\left( tx+\left( 1-t\right) y\right) g(tx+\left(
1-t\right) y)dtdxdy \\
&\leq &\left( \frac{1}{b-a}\dint\limits_{a}^{b}f(x)g(x)dx\right) \left(
\dint\limits_{0}^{1}\left( h^{s}(t^{2})+h^{s}(t-t^{2})\right) dt\right) \\
&&+f\left( \frac{a+b}{2}\right) g\left( \frac{a+b}{2}\right) \left(
\dint\limits_{0}^{1}\left( h^{s}\left( \left( 1-t\right) ^{2}\right)
+h^{s}(t-t^{2})\right) dt\right)
\end{eqnarray*}%
for all $s\in \left( 0,1\right] ,$ $x,y\in I\subseteq
\mathbb{R}
.$
\end{theorem}

\begin{proof}
Since $f,g$ are $\left( h-s\right) _{2}-$convex functions on $I$, we have%
\begin{eqnarray*}
f\left( tx+\left( 1-t\right) \frac{a+b}{2}\right) &\leq
&h^{s}(t)f(x)+h^{s}\left( 1-t\right) f\left( \frac{a+b}{2}\right) \\
g\left( tx+\left( 1-t\right) \frac{a+b}{2}\right) &\leq
&h^{s}(t)g(x)+h^{s}\left( 1-t\right) g\left( \frac{a+b}{2}\right) .
\end{eqnarray*}%
By multiplying the above inequalities and since $h$ is
super-multiplicative
function, we get%
\begin{eqnarray*}
&&f\left( tx+\left( 1-t\right) \frac{a+b}{2}\right) g\left( tx+\left(
1-t\right) \frac{a+b}{2}\right) \\
&\leq &h^{s}(t^{2})f(x)g(x)+h^{s}\left( \left( 1-t\right) ^{2}\right)
f\left( \frac{a+b}{2}\right) g\left( \frac{a+b}{2}\right) \\
&&+h^{s}(t-t^{2})\left[ f\left( \frac{a+b}{2}\right) g(x)+f(x)g\left( \frac{%
a+b}{2}\right) \right] .
\end{eqnarray*}%
Since $f$ and $g$ are similarly ordered functions, we get%
\begin{eqnarray*}
&&f\left( tx+\left( 1-t\right) y\right) g(tx+\left( 1-t\right) y) \\
&\leq &\left( h^{s}(t^{2})+h^{s}(t-t^{2})\right) f(x)g(x)+\left( h^{s}\left(
\left( 1-t\right) ^{2}\right) +h^{s}(t-t^{2})\right) f\left( \frac{a+b}{2}%
\right) g\left( \frac{a+b}{2}\right) .
\end{eqnarray*}%
By integrating the resulting inequality with respect to $t$ over $\left[ 0,1%
\right] $ and with respect to $x,y$ over $\left[ a,b\right] \times \left[ a,b%
\right] ,$ we have%
\begin{eqnarray*}
&&\frac{1}{b-a}\dint\limits_{a}^{b}\dint\limits_{a}^{b}\dint\limits_{0}^{1}f%
\left( tx+\left( 1-t\right) y\right) g(tx+\left( 1-t\right) y)dtdxdy \\
&\leq &\left( \dint\limits_{a}^{b}f(x)g(x)dx\right) \left(
\dint\limits_{0}^{1}\left( h^{s}(t^{2})+h^{s}(t-t^{2})\right) dt\right) \\
&&+\left( b-a\right) f\left( \frac{a+b}{2}\right) g\left( \frac{a+b}{2}%
\right) \left( \dint\limits_{0}^{1}\left( h^{s}\left( \left( 1-t\right)
^{2}\right) +h^{s}(t-t^{2})\right) dt\right) .
\end{eqnarray*}%
By dividing $\frac{1}{b-a},$ the proof is completed.
\end{proof}

\begin{theorem}
\label{t6}Let $f,g\in SX(\left( h-s\right) _{2},I)$, $h$ is
super-multiplicative and $f,g$ be similarly ordered functions on $I$. If $%
f,g,fg\in L_{1}\left[ a,b\right] $ and $h\in L_{1}\left[ 0,1\right] ,$ then
we have the following inequalities;%
\begin{eqnarray}
&&f\left( \frac{a+b}{2}\right) g\left( \frac{a+b}{2}\right)  \label{2.6a} \\
&\leq &2h^{s}\left( \frac{1}{4}\right) M\left( a,b\right)
\dint\limits_{0}^{1}\left( h^{s}\left( t\right) +h^{s}\left( 1-t\right)
\right) ^{2}dt.  \notag
\end{eqnarray}%
and%
\begin{eqnarray}
&&\frac{1}{2h^{2s}\left( \frac{1}{2}\right) }f\left( \frac{a+b}{2}\right)
g\left( \frac{a+b}{2}\right)  \label{2.6b} \\
&\leq &\frac{1}{b-a}\dint\limits_{a}^{b}f\left( x\right) g\left( x\right) dx+%
\frac{M\left( a,b\right) }{2}\left( \dint\limits_{0}^{1}\left( h^{s}\left(
t\right) +h^{s}\left( 1-t\right) \right) ^{2}dt\right)  \notag
\end{eqnarray}%
for all $s\in \left( 0,1\right] .$
\end{theorem}

\begin{proof}
Since $f$ and $g$ are $\left( h-s\right) _{2}-$convex on $[a,b]$, then for $%
t\in \lbrack a,b],$ we observe that%
\begin{eqnarray}
&&f\left( \frac{a+b}{2}\right) g\left( \frac{a+b}{2}\right)  \label{k} \\
&=&f\left( \frac{ta+\left( 1-t\right) b}{2}+\frac{tb+\left( 1-t\right) a}{2}%
\right) g\left( \frac{ta+\left( 1-t\right) b}{2}+\frac{tb+\left( 1-t\right) a%
}{2}\right)  \notag \\
&\leq &h^{2s}\left( \frac{1}{2}\right) \left[ \left( f\left( ta+\left(
1-t\right) b\right) +f\left( tb+\left( 1-t\right) a\right) \right) \right]
\notag \\
&&\times \left[ \left( g\left( ta+\left( 1-t\right) b\right) +g\left(
tb+\left( 1-t\right) a\right) \right) \right] .  \notag
\end{eqnarray}%
By using $\left( h-s\right) _{2}-$convexity of $f$ and $g,$ we get%
\begin{eqnarray*}
&&f\left( \frac{a+b}{2}\right) g\left( \frac{a+b}{2}\right) \\
&\leq &h^{2s}\left( \frac{1}{2}\right) \left[ \left( h^{s}\left( t\right)
f\left( a\right) +h^{s}\left( 1-t\right) f\left( b\right) +h^{s}\left(
t\right) f\left( b\right) +h^{s}\left( 1-t\right) f\left( a\right) \right) %
\right] \\
&&\times \left[ \left( h^{s}\left( t\right) g\left( a\right) +h^{s}\left(
1-t\right) g\left( b\right) +h^{s}\left( t\right) g\left( b\right)
+h^{s}\left( 1-t\right) g\left( a\right) \right) \right] .
\end{eqnarray*}%
Since $h$ is super-multiplicative, we have%
\begin{eqnarray*}
&&f\left( \frac{a+b}{2}\right) g\left( \frac{a+b}{2}\right) \\
&\leq &h^{s}\left( \frac{1}{4}\right) \left( h^{s}\left( t^{2}\right)
+2h^{s}\left( t-t^{2}\right) +h^{s}\left( \left( 1-t\right) ^{2}\right)
\right) \\
&&\times \left[ f\left( a\right) g\left( a\right) +f\left( b\right) g\left(
b\right) +f\left( a\right) g\left( b\right) +f\left( b\right) g\left(
a\right) \right] .
\end{eqnarray*}%
By using the similarly ordered property of $f$ and $g,$ we obtain
\begin{eqnarray*}
&&f\left( \frac{a+b}{2}\right) g\left( \frac{a+b}{2}\right) \\
&\leq &2h^{s}\left( \frac{1}{4}\right) \left( h^{s}\left( t\right)
+h^{s}\left( 1-t\right) \right) ^{2}\left[ f\left( a\right) g\left( a\right)
+f\left( b\right) g\left( b\right) \right] .
\end{eqnarray*}%
By integrating both sides respect to $t$ over the interval $\left[ 0,1\right]
,$ we get%
\begin{eqnarray*}
&&f\left( \frac{a+b}{2}\right) g\left( \frac{a+b}{2}\right) \\
&\leq &2h^{s}\left( \frac{1}{4}\right) \left[ f\left( a\right) g\left(
a\right) +f\left( b\right) g\left( b\right) \right] \dint\limits_{0}^{1}%
\left( h^{s}\left( t\right) +h^{s}\left( 1-t\right) \right) ^{2}dt.
\end{eqnarray*}%
Which completes the proof of the first inequality. Therefore, from the
inequality (\ref{k}), we can write%
\begin{eqnarray*}
&&f\left( \frac{a+b}{2}\right) g\left( \frac{a+b}{2}\right) \\
&\leq &h^{2s}\left( \frac{1}{2}\right) \left[ \left( f\left( ta+\left(
1-t\right) b\right) +f\left( tb+\left( 1-t\right) a\right) \right) \right] \\
&&\times \left[ \left( g\left( ta+\left( 1-t\right) b\right) +g\left(
tb+\left( 1-t\right) a\right) \right) \right] \\
&\leq &h^{2s}\left( \frac{1}{2}\right) \left[ f\left( ta+\left( 1-t\right)
b\right) g\left( ta+\left( 1-t\right) b\right) \right. \\
&&\left. +f\left( tb+\left( 1-t\right) a\right) g\left( tb+\left( 1-t\right)
a\right) \right] \\
&&+h^{2s}\left( \frac{1}{2}\right) \left[ \left[ h^{s}\left( t\right)
f\left( a\right) +h^{s}\left( 1-t\right) f\left( b\right) \right] \left[
h^{s}\left( t\right) g\left( b\right) +h^{s}\left( 1-t\right) g\left(
a\right) \right] \right. \\
&&\left. +\left[ h^{s}\left( t\right) f\left( b\right) +h^{s}\left(
1-t\right) f\left( a\right) \right] \left[ h^{s}\left( t\right) g\left(
a\right) +h^{s}\left( 1-t\right) g\left( b\right) \right] \right] .
\end{eqnarray*}%
Now by using the similar ordered property of $f$ and $g,$ by integrating the
resulting inequality, we have%
\begin{eqnarray*}
&&f\left( \frac{a+b}{2}\right) g\left( \frac{a+b}{2}\right) \\
&\leq &h^{2s}\left( \frac{1}{2}\right) \dint\limits_{0}^{1}\left[ f\left(
ta+\left( 1-t\right) b\right) g\left( ta+\left( 1-t\right) b\right) \right.
\\
&&\left. +f\left( tb+\left( 1-t\right) a\right) g\left( tb+\left( 1-t\right)
a\right) \right] dt \\
&&+h^{2s}\left( \frac{1}{2}\right) \left[ f\left( a\right) g\left( a\right)
+f\left( b\right) g\left( b\right) \right] \left( \dint\limits_{0}^{1}\left(
h^{s}\left( t\right) +h^{s}\left( 1-t\right) \right) ^{2}dt\right) .
\end{eqnarray*}%
Changing of the variable, we obtain%
\begin{eqnarray*}
&&f\left( \frac{a+b}{2}\right) g\left( \frac{a+b}{2}\right) \\
&\leq &\frac{2h^{2s}\left( \frac{1}{2}\right) }{b-a}\dint\limits_{a}^{b}f%
\left( x\right) g\left( x\right) dx \\
&&+h^{2s}\left( \frac{1}{2}\right) \left[ f\left( a\right) g\left( a\right)
+f\left( b\right) g\left( b\right) \right] \left( \dint\limits_{0}^{1}\left(
h^{s}\left( t\right) +h^{s}\left( 1-t\right) \right) ^{2}dt\right) .
\end{eqnarray*}%
Dividing both sides of the resulting inequality by $2h^{2s}\left( \frac{1}{2}%
\right) ,$ we get the second inequality.
\end{proof}

\begin{remark}
\bigskip In Theorem \ref{t6}, if we take $h\left( t\right) =t,$ then we get%
\begin{eqnarray*}
&&f\left( \frac{a+b}{2}\right) g\left( \frac{a+b}{2}\right) \\
&\leq &2^{1-2s}M\left( a,b\right) \left( \frac{1}{1+2s}+\frac{\sqrt{\pi }}{%
2^{2s}}\frac{\Gamma \left( 1+s\right) }{\Gamma \left( \frac{3}{2}+s\right) }+%
\frac{2^{2s}\sqrt{2}}{2^{2s-1/2}\left( 2+4s\right) }\right) \\
&=&M\left( a,b\right) \left( \frac{4}{2^{2s}\left( 1+2s\right) }+\frac{2%
\sqrt{\pi }}{2^{4s}}\frac{\Gamma \left( 1+s\right) }{\Gamma \left( \frac{3}{2%
}+s\right) }\right)
\end{eqnarray*}%
and%
\begin{eqnarray*}
&&2^{2s-1}f\left( \frac{a+b}{2}\right) g\left( \frac{a+b}{2}\right) \\
&\leq &\frac{1}{b-a}\dint\limits_{a}^{b}f\left( x\right) g\left( x\right) dx+%
\frac{M\left( a,b\right) }{2}\dint\limits_{0}^{1}\left( t^{s}+\left(
1-t\right) ^{s}\right) ^{2}dt \\
&=&\frac{1}{b-a}\dint\limits_{a}^{b}f\left( x\right) g\left( x\right) dx+%
\frac{M\left( a,b\right) }{2}\left( \frac{2}{1+2s}+\frac{\sqrt{\pi }}{2^{2s}}%
\frac{\Gamma \left( 1+s\right) }{\Gamma \left( \frac{3}{2}+s\right) }\right)
\end{eqnarray*}%
for all $s\in \left( 0,1\right] .$
\end{remark}

\begin{theorem}
\label{t7}Let $f,g\in SX(\left( h-s\right) _{2},I)$, $h$ is
super-multiplicative and $f,g$ be similarly ordered functions on $I$. If $%
f,g,fg\in L_{1}\left[ a,b\right] $ and $h\in L_{1}\left[ 0,1\right] ,$ then
we have the following inequalities;%
\begin{eqnarray*}
&&\frac{g\left( b\right) }{b-a}\int\limits_{a}^{b}h^{s}\left( \frac{x-a}{b-a}%
\right) f\left( x\right) dx+\frac{g\left( a\right) }{b-a}\int%
\limits_{a}^{b}h^{s}(\frac{b-x}{b-a})f\left( x\right) dx \\
&&+\frac{f\left( b\right) }{b-a}\int\limits_{a}^{b}h^{s}(\frac{x-a}{b-a}%
)g\left( x\right) dx+\frac{f\left( a\right) }{b-a}\int\limits_{a}^{b}h^{s}(%
\frac{b-x}{b-a})g\left( x\right) dx \\
&\leq &\frac{1}{b-a}\int\limits_{a}^{b}f\left( x\right) g\left( x\right)
dx+f\left( b\right) g\left( b\right) \int\limits_{0}^{1}\left[ h^{s}\left(
t-t^{2}\right) +h^{s}\left( t^{2}\right) \right] dt \\
&&+f\left( a\right) g\left( a\right) \int\limits_{0}^{1}\left[ h^{s}\left(
\left( 1-t\right) ^{2}\right) +h^{s}\left( t-t^{2}\right) \right] dt.
\end{eqnarray*}%
for all $s\in \left( 0,1\right] .$
\end{theorem}

\begin{proof}
Since $f$ and $g$ are $\left( h-s\right) _{2}-$convex functions, we can write%
\begin{equation*}
f\left( tb+\left( 1-t\right) a\right) \leq h^{s}\left( t\right) f\left(
b\right) +h^{s}\left( 1-t\right) f\left( a\right)
\end{equation*}%
and%
\begin{equation*}
g\left( tb+\left( 1-t\right) a\right) \leq h^{s}\left( t\right) g\left(
b\right) +h^{s}\left( 1-t\right) g\left( a\right)
\end{equation*}%
By using the elementary inequality, $e\leq f$ and $p\leq r$, then $er+fp\leq
ep+fr$ for $e,f,p,r\in
\mathbb{R}
,$ then we get%
\begin{eqnarray*}
&&f\left( tb+\left( 1-t\right) a\right) \left[ h^{s}\left( t\right) g\left(
b\right) +h^{s}\left( 1-t\right) g\left( a\right) \right] \\
&&+g\left( tb+\left( 1-t\right) a\right) \left[ h^{s}\left( t\right) f\left(
b\right) +h^{s}\left( 1-t\right) f\left( a\right) \right] \\
&\leq &f\left( tb+\left( 1-t\right) a\right) g\left( tb+\left( 1-t\right)
a\right) \\
&&+\left[ h^{s}\left( t\right) f\left( b\right) +h^{s}\left( 1-t\right)
f\left( a\right) \right] \left[ h^{s}\left( t\right) g\left( b\right)
+h^{s}\left( 1-t\right) g\left( a\right) \right] .
\end{eqnarray*}%
So, we obtain%
\begin{eqnarray*}
&&h^{s}\left( t\right) f\left( tb+\left( 1-t\right) a\right) g\left(
b\right) +h^{s}\left( 1-t\right) f\left( tb+\left( 1-t\right) a\right)
g\left( a\right) \\
&&+h^{s}\left( t\right) f\left( b\right) g\left( tb+\left( 1-t\right)
a\right) +h^{s}\left( 1-t\right) f\left( a\right) g\left( tb+\left(
1-t\right) a\right) \\
&\leq &f\left( tb+\left( 1-t\right) a\right) g\left( tb+\left( 1-t\right)
a\right) \\
&&+\left[ h^{2s}\left( t\right) f\left( b\right) g(b)+h^{s}\left( t\right)
h^{s}\left( 1-t\right) f(b)g(a)+h^{s}\left( t\right) h^{s}\left( 1-t\right)
f\left( a\right) g(b)+h^{2s}\left( 1-t\right) f(a)g(a)\right] .
\end{eqnarray*}%
By using similarly ordered property of $f$ and $g,$ we have%
\begin{eqnarray*}
&&h^{s}\left( t\right) f\left( tb+\left( 1-t\right) a\right) g\left(
b\right) +h^{s}\left( 1-t\right) f\left( tb+\left( 1-t\right) a\right)
g\left( a\right) \\
&&+h^{s}\left( t\right) f\left( b\right) g\left( tb+\left( 1-t\right)
a\right) +h^{s}\left( 1-t\right) f\left( a\right) g\left( tb+\left(
1-t\right) a\right) \\
&\leq &f\left( tb+\left( 1-t\right) a\right) g\left( tb+\left( 1-t\right)
a\right) \\
&&+\left[ h^{2s}\left( 1-t\right) +h^{s}\left( t\right) h^{s}\left(
1-t\right) \right] f(a)g(a)+\left[ h^{s}\left( t\right) h^{s}\left(
1-t\right) +h^{2s}\left( t\right) \right] f\left( b\right) g(b).
\end{eqnarray*}%
Since $h$ is super-multiplicative, we can write%
\begin{eqnarray*}
&&h^{s}\left( t\right) f\left( tb+\left( 1-t\right) a\right) g\left(
b\right) +h^{s}\left( 1-t\right) f\left( tb+\left( 1-t\right) a\right)
g\left( a\right) \\
&&+h^{s}\left( t\right) f\left( b\right) g\left( tb+\left( 1-t\right)
a\right) +h^{s}\left( 1-t\right) f\left( a\right) g\left( tb+\left(
1-t\right) a\right) \\
&\leq &f\left( tb+\left( 1-t\right) a\right) g\left( tb+\left( 1-t\right)
a\right) \\
&&+\left[ h^{s}\left( \left( 1-t\right) ^{2}\right) +h^{s}\left(
t-t^{2}\right) \right] f(a)g(a)+\left[ h^{s}\left( t-t^{2}\right)
+h^{s}\left( t^{2}\right) \right] f\left( b\right) g(b).
\end{eqnarray*}%
By integrating this inequality with respect to $t$ over $\left[ 0,1\right] $
and by using the change of the variable $tb+\left( 1-t\right) a=x,$ $%
(b-a)dt=dx,$ the proof is completed.
\end{proof}

\bigskip

\section{\protect\bigskip Applications to Some Special Means}

\bigskip We shall consider the means as arbitrary positive real numbers $%
a,b, $ $a\neq b.$

The geometric mean:
\begin{equation*}
G=G\left( a,b\right) :=\sqrt{ab},\text{ \ }a,b\geq 0,
\end{equation*}

The Identric mean.

\begin{equation*}
I=I\left( a,b\right) :=\left\{
\begin{array}{l}
a\text{ \ \ \ \ \ \ \ \ \ \ \ \ \ \ \ \ \ if \ \ }a=b \\
\frac{1}{e}\left( \frac{b^{b}}{a^{a}}\right) ^{\frac{1}{b-a}}\text{ \ \ \ \
\ if \ \ }a\neq b%
\end{array}%
\right. ,\text{ \ }a,b\geq 0,
\end{equation*}

Now, we present some applications of the result in section 2 to the special
means of real numbers. The following propositions hold:

\begin{proposition}
\label{p1}\bigskip \textit{Let }$2<a<b<\infty .$\textit{\ Then for all }$%
s\in (0,1],$ \textit{we have}%
\begin{equation*}
I(a,b)\leq G^{\frac{2}{s+1}}(a,b).
\end{equation*}
\end{proposition}

\begin{proof}
The proof is immediate follows from Theorem \ref{t1} applied for $%
f(x)=g(x)=\ln x$; $x\in \lbrack 2;\infty )$ and $h(t)=t$. In other words,
the following is obtained.%
\begin{eqnarray*}
&&\frac{1}{b-a}\dint\limits_{a}^{b}lnxdx=\ln I(a,b) \\
&\leq &\frac{\ln b}{\left( b-a\right) ^{2+s}}\dint\limits_{a}^{b}\left(
x-a\right) \left( x-a\right) ^{s}dx+\frac{\ln a}{\left( b-a\right) ^{2+s}}%
\dint\limits_{a}^{b}\left( x-a\right) \left( b-x\right) ^{s}dx \\
&&+\frac{\ln b}{\left( b-a\right) ^{2+s}}\dint\limits_{a}^{b}\left(
b-x\right) \left( x-a\right) ^{s}dx+\frac{\ln a}{\left( b-a\right) ^{2+s}}%
\dint\limits_{a}^{b}\left( b-x\right) \left( b-x\right) ^{s}dx \\
&=&\frac{\ln a+\ln b}{s+2}+\frac{\ln a+\ln b}{\left( s+1\right) \left(
s+2\right) }=\frac{\ln a+\ln b}{s+1}=\ln G^{\frac{2}{s+1}}(a,b)
\end{eqnarray*}%
This completes the proof.
\end{proof}

\bigskip \bigskip

Is similar to the above features, may be different applications. We stay
away from getting into these details here. Interested readers can add new
ones to these applications. The theory of convex functions is expanding day
by day. The reason, researchers are opening a new chapter in the theory of
convex functions with each passing day. Convex function is obtained as a
result of studies of different classes often resemble each other in terms of
several properties. Covers all of the previous definitions of new classes
sometimes obtained, and sometimes carries a few features.

\end{document}